\documentclass[12pt]{article}
\title{}
\author{}
\date{}
\renewcommand\baselinestretch{1}

\def\Irr{\hbox{\rm Irr}}
\def\Aut{\hbox{\rm Aut}}
\def\NL{\hbox{\rm NL}}
\def\HT{\heiti\bf\relax}
\def\ST{\songti\rm\relax}

\newcommand{\vf}[1]
{\frac{\partial V}{\partial x}f_#1(x)}

\renewcommand{\theequation}{\thesection.\arabic{equation}}     

\usepackage{amssymb,amsthm}
\usepackage{amsmath}

\newtheorem*{main theorem}{Main Theorem}
\newtheorem{theorem}{Theorem}[section]
\newtheorem{corollary}[theorem]{Corollary}

\newtheorem{lemma}[theorem]{Lemma}

\newcommand{\cd}[1]{{\rm  cd}(#1)}

\begin{document}
\pagestyle{empty}                                              

\newpage
\pagestyle{plain}
\pagenumbering{arabic}                                         
\renewcommand\baselinestretch{1}
\begin{center}
\huge
Square character degree graphs yield direct products
\end{center}
\vspace{0.3cm}
\begin{center}
{\Large
Mark\  L.\  Lewis} \\
Department of Mathematical Sciences \\
Kent State University \\
Kent, OH 44242 \\
e-mail: lewis@math.kent.edu \\

\medskip

{\Large Qingyun\  Meng }\\
School of Mathematical Sciences \\
Xiamen University \\
Xiamen, China 361005 \\
e-mail: meng19821030@live.cn \\
\end{center}
\vspace{0.7cm}
\noindent{\small Abstract.}
{\small If $G$ is a solvable group, we take $\Delta (G)$ to be the character degree graph for $G$ with primes as vertices. We prove that if $\Delta (G)$ is a square, then $G$ must be a direct product.}


\smallskip
\noindent{\small CLC number }\ \  O152.1\ \ \ \ {\small Document Code }\ \ A

\pagestyle{plain}
\pagenumbering{arabic}
\section{Introduction} \label{section 1}


Throughout this paper, all groups are finite.  Let $G$ be a group. The set of irreducible character degrees of $G$ is denoted by $\cd G$.  The character degree graph of $G$, written $\Delta (G)$, has a vertex set, $\rho (G)$, that consists of the primes that divide degrees in $\cd G $.  There is an edge between $p$ and $q$ if $pq$ divides some degree $a \in \cd G$.

Character degree graphs have proven to be a useful tool to study the structure of $G$ when given information regarding $\Irr (G)$.  They have been studied more then twenty years, and people have obtained a number of interesting results.  For example, for a finite solvable group $G$, the graph $\Delta(G)$ has at most two connected components.  In addition, if $\Delta(G)$ is disconnected, then each connected component is a complete graph.  Moreover, P\'{a}lfy has proved that for a finite solvable group $G$, any three primes in $\rho(G)$ must have an edge in $\Delta(G)$ that is incident to two of those primes. We will call this the Three Primes theorem.  We will prove that one consequence of the Three Primes theorem, is that if $\Delta (G)$ has at least four vertices, then either
$\Delta (G)$ contains a triangle (i.e., a complete subgraph of three vertices) or $\Delta (G)$ is a square.

It is not difficult  to find solvable groups where $\Delta (G)$ is a square.  These groups have been studied by the first author in \cite{[1]} where it was proved that they have Fitting height at most $4$.  Let $G$ be a finite solvable group with $\Gamma$ as its character degree graph, where $\Gamma$ is a square with $\rho = \rho (G) = \{p, q, r, s \}$ as its vertex set, and the set $\{pr, ps, qr, qs \}$ as its edge set.  One natural way to construct a group $G$ with this structure is to take $G = A \times B$ where $\Delta (A)$ is disconnected with components $\{ p \}$ and $\{ q \}$ and $\Delta (B)$ is disconnected with components $\{ r \}$ and $\{ s \}$.  The question that arises is whether there are any other ways to obtain solvable groups with this graph.  We prove that in fact there are not.

\begin{main theorem}
Let $G$ be a solvable group where $\Delta (G) = \Gamma$.  Then $G = A \times B$ where $\rho (A) = \{ p , q \}$ and $\rho (B) = \{ r, s \}$.
\end{main theorem}

This theorem is a generalization for solvable groups of Theorem B of \cite{struct}.  That theorem stated that if $p$, $q$, $r$, and $s$ were distinct primes and $G$ is a group with
$\cd G = \{ 1, p, q, r, s, pr, ps, qr, qs \}$, then $G = A \times B$ with $\cd A = \{ 1, p, q \}$ and $\cd B = \{ 1, r, s \}$.  It is easy to see that $\Delta (G)$ is $\Gamma$, so the Main Theorem applies, and it is not difficult to see that the conclusion implies this result.

We strongly believe that the hypothesis that $G$ is solvable can be removed.  However, it is likely that this would require appealing to the classification of simple groups whose orders are divisible by at most $4$ primes found in \cite{HuLe}.  The arguments employed will likely be of a highly different flavor, and we have not pursued this at this time.

The proof of the Main Theorem will be broken into two pieces
depending on whether or not $G$ has a normal nonabelian Sylow
subgroup for some prime.  If there exists a prime $t \in \rho(G)$
such that $G$ has a normal Sylow $t$-subgroup, then we study the
structure of $G$ in Section \ref{section 2}.  In Section
\ref{section 3}, we study the case where $G$ has no nonabelian Sylow
subgroups.  In this case, the degree graph of $G/\Phi(G)$ also is
$\Gamma$, so that we can consider $G/\Phi(G)$.  We then extend the
results for $G/\Phi (G)$ to $G$ in two steps depending whether or
not ${\bf F} (G)$ is abelian.

\section{Character degree graphs}

We first establish some notation which will be used repeatedly.  If $m$ is an integer, then $\pi (m)$ is the set of primes that divide $m$.  If $G$ is a group, then $\pi (G) = \pi (|G|)$ and if $N$ is a normal subgroup of $G$, then $\pi (G:N) = \pi (|G:N|)$.

Let $\NL(G)$ denote the set of nonlinear irreducible characters of
$G$. If $N$ is a normal subgroup of $G$, then $\Irr (G \mid N)$ is
the set of irreducible characters of $G$ whose kernels do not
contain $N$. Therefore, $\Irr (G)$ is a disjoint union of $\Irr (G
\mid N)$ and $\Irr(G/N)$. Define $\cd {G \mid N} = \{\chi(1) \mid
\chi \in \Irr(G \mid N) \}$ and observe that $\cd G = \cd {G/N} \cup
\cd {G \mid N}$.  Also, for a character $\theta \in \Irr (N)$, we
use the usual notation that $\Irr (G \mid \theta)$ is the set of
irreducible constituents of $\theta^{G}$, and we define $\cd {G \mid
\theta} = \{\chi(1) \mid \chi \in \Irr (G \mid \theta) \}$.  Note
that $\cd {G \mid N}$ is the union of the sets $\cd {G \mid \theta}$
as $\theta$ runs through all the nonprincipal characters in $\Irr
(N)$.

We now prove an assertion made in the Introduction.

\begin{lemma} \label{no triangles}
Let $G$ be a solvable group.  If $\Delta (G)$ has at least $4$ vertices, then either $\Delta (G)$ contains a triangle or $\Delta (G)$ is a square.
\end{lemma}

\begin{proof}
Suppose $\Delta (G)$ has no triangles and has at least four vertices.  We prove that $\Delta (G)$ is a square.  We make strong use of the three prime condition.  If $\Delta (G)$ is disconnected, then three prime condition implies that each connected component is a complete graph.  The condition of no triangles implies that each connected component has at most two vertices.  P\'alfy proved the disconnected graph where each connected component has two vertices cannot occur.  (See \cite {[9]}.)  Thus, $\Delta (G)$ must be connected.  If a vertex has $3$ or more neighbors, then the three prime condition implies that there will be edge incident to two of these neighbors, and $\Delta (G)$ will have a triangle.  Thus, the condition of no triangles implies that every vertex of $\Delta (G)$ has degree at most $2$.  I.e., $\Delta (G)$ will be either a path or a cycle.  The three prime condition disallows paths with $5$ or more vertices and cycles with $6$ or more vertices. The path with $4$ vertices was proved to not occur by Zhang in \cite{[2]}, and the cycle with $5$ vertices was proved to
not exist by the first author in \cite{[10]}.  This leaves only the cycle with $4$ vertices, i.e. a square.
\end{proof}

To study groups where $\Delta (G)$ is a square, we will come across many examples of solvable groups whose degree graphs are disconnected.  These groups have been studied extensively.  We will make use of the classification of these groups that can be found in \cite{[5]}.  We will say that $G$ is disconnected if $\Delta (G)$ is disconnected.  We will say that $G$ is of disconnected Type $n$ if $G$ satisfies the hypotheses of Example 2.n in \cite{[5]}.  We give a brief summary of the some of the facts regarding six types of disconnected groups.  We are not giving full descriptions.  For full description one should consult \cite{[5]}.

We say that $G$ is disconnected of Type 1, if $G$ has a normal nonabelian Sylow $p$-subgroup $P$ and an abelian $p$-complement $H$.  Also, $P$ has nilpotence class 2.

For $G$ disconnected of Types 2 and 3, $G$ is a semi-direct product
of a group $H$ acting on a group $P$ where $|P| = 9$.  Let $Z = {\bf
C}_H (P) = {\bf Z} (G)$.  If $G$ is Type 2, then $H/Z \cong  {\rm SL}_2
(3)$ and if $G$ is Type 3, then $H/Z \cong  {\rm GL}_2 (3)$. In both
cases, $\rho (G) = \{ 2, 3 \}$ and ${\bf F}(G) = P \times Z$.  If we
write $F$ and $E/F$ for the Fitting subgroups of $G$ and $G/F$, we
see that $E/F$ is isomorphic to the quaternions, and in particular,
$E/F$ is not abelian.

If $G$ is disconnected of Type 4, 5, and 6, then we take $F$ and
$E/F$ to be the Fitting subgroups of $G$ and $G/F$, and $Z = {\bf Z}
(G)$. It is known that $G/E$ and $E/F$ are cyclic.  When $G$ is of
Type 4, $[E,F]$ is a minimal normal subgroup of $G$ and $F = [E,F]
\times Z$, and the two connected components of $\Delta (G)$ are $\pi
(G:E)$ and $\pi (E:F)$.  Also, there is a prime power $q$ so that
$|[E,F]| = q^m$ where $m = |G:E|$, and $(q^m - 1)/(q - 1)$ divides
$|E:F|$.

When $G$ is disconnected of Type 5, then $F = Q \times Z$ where $Q$ is a nonabelian $2$-group.  Also, $|G:E| = 2$, and the two connected components of $\Delta (G)$ are $\{ 2 \}$ and $\pi (E:F)$.

Finally, if $G$ is disconnected of Type 6, then $G$ has normal nonabelian Sylow $p$-subgroup $P$ and $F = P \times Z$.  The two connected components of $\Delta (G)$ are $\{ p \} \cup \pi (|E:F|)$ and $\pi (|G:E|)$.  In particular, if $G$ is of Type 6, then $\rho (G)$ contains at least three primes.

We summarize some of the facts we need.  In particular, if $G$ is of
Types 2, 3, or 4, then $G$ has an abelian Fitting subgroup, and in
Types 1, 5, and 6, $G$ has a nonabelian Fitting subgroup.  In Types
2 and 3, the Sylow $2$-subgroup of $G$ is nonabelian.  Also, in
Types 1, 4, 5, and 6, ${\bf F}(G/{\bf F}(G))$ is abelian, whereas in
Types 2 and 3, ${\bf F} (G/{\bf F}(G))$ is nonabelian.

We prove the following number theory fact which uses the Zsigmondy
prime theorem.

\begin{lemma} \label{zsig}
Let $p$ be a prime and let $a \ge 1$ and $n > 1$ be integers.  If $(p^{an}-1)/(p^a-1)$ is a power of a prime, then $n$ is a prime and either $a$ is a power of $n$ or $p = 2$, $a = 3$, and $n = 2$.
\end{lemma}

\begin{proof}
Let $\Phi_d (x)$ be the cyclotomic polynomial for $d$.  We know that $(p^{an} - 1)/(p^a - 1) = \prod \Phi_d (p)$ where $d$ runs over the divisors of $an$ that do not divide $a$.  The Zsigmondy prime theorem says that $\Phi_d (p)$ is divisible by a prime that does not divide $\Phi_e (p)$ for all $e < d$ except if $p$ is a Mersenne prime and $d = 2$ or if $p = 2$ and $d = 6$.  In particular, if $n$ is not prime, then $\Phi_{an} (p)$ will properly divide $(p^{an}-1)/(p^a - 1)$ and unless $a = 1$, $n = 6$, and $p = 2$, $\Phi_{an} (p)$ will have a prime divisor that does not divide the other factors contradicting the fact that $(p^{an}-1)/(p-1)$ is a prime power.  Note that if $(2^6 - 1)/(2^1 - 1) = 63$ is not a prime power.  Thus, $n$ must be a prime.

Now, write $a = bc$ where $b$ is a power of $n$ and $c$ is not divisible by $n$.  If $c > 1$, then $\Phi_{an} (p)$ and $\Phi_{bn} (p)$ both divide $(p^{an}-1)/(p^a-1)$ and are not equal.  If we do not have $p = 2$ and $an = 6$, then we see that $\Phi_{an} (p)$ has a prime divisor that does not divide $\Phi_{bn} (p)$ and $(p^{an}-1)/(p^a-1)$ will not be a prime power.  Suppose now that $p = 2$ and $an = 6$.  If $a = 2$ and $n = 3$, then $(2^6-1)/(2^2-1) = 63/3 = 21$ is not a prime power.  Hence, we must have $a = 3$ and $n = 2$.  We note that this is a real exception since $(2^6-1)/(2^3-1) = 63/7 = 9$ is a prime power.
\end{proof}

We now apply this to disconnected groups of Type 4.  This lemma is related to Lemma 4.2 of \cite{struct} and Lemma 4.3 of \cite{[5]}.

\begin{lemma} \label{not square}
Let $G$ be a disconnected group of Type 4.  Let $F$ and $E/F$ be the Fitting subgroups of $G$ and $G/F$.  Suppose that $\pi (E:F) = \{ r \}$.  Then $|G:E| = n$ is a prime. If the prime $p$ dividing $|[E,F]|$ is not $n$, then $n$ is odd and $|[E,F]| = p^m$ where $m$ is a power of $n$.  In particular, $|[E,F]|$ is not a square.
\end{lemma}

\begin{proof}
We know that $|[E,F]| = (p^a)^n$ where $p$ is a prime and $a$ is a positive integer.  We know that $n > 1$.  We also know that $(p^{an}-1)/(p^a-1)$ divides $|E:F|$.  Since $\pi (E:F) = \{ r \}$, it follows that $|E:F|$ and hence $(p^{an}-1)/(p^a-1)$ are powers of $r$.  By Lemma \ref{zsig}, we know that $n$ is a prime, and either $a$ is a power of $n$ or $p = 2$, $a = 3$, and $n = 2$.  We now suppose that $n \ne p$.  Thus, we have that $a$ is a power of $n$, so $m = an$ is a power of $n$.  If $n = 2$, then $(p^{a2} - 1)/(p^a-1) = p^a + 1$ is even.  This implies that $r = 2$ which contradicts the fact that $\Delta (G)$ is disconnected.  Thus, $n$ is odd.
\end{proof}

\section{Normal nonabelian Sylow subgroups} \label{section 2}

We start by stating the hypothesis that we study throughout the rest of this paper.

\medskip
{\bf Hypothesis 1.}  Let $G$ be a solvable group with $\Gamma$ as
its character degree graph, where $\Gamma$ is a square with $ \rho
(G)= \rho =\{p, q, r, s\}$ as its vertex set, and the set $\{pr, ps,
qr, qs \}$ as its edge set.
We also set some notation that we use for $G$.  By It\^o's theorem, $G$ has an abelian normal Hall $\rho^{'}$-subgroup, say $A$.  Write $F$ and $E/F$ for the Fitting subgroups of $G$ and $G/F$, respectively.

\medskip
This next lemma shows that if a subgraph of $\Gamma$ has the same vertex set as $\Gamma$ and is $\Delta (H)$ for some solvable group $H$, then $\Delta (H)$ must be $\Gamma$.

\begin{lemma}\label{Lemma 1}
Assume hypothesis 1.  No proper subgraph of $\Gamma$ with four vertices can be the degree graph of a finite solvable group.
\end{lemma}

\begin{proof}
Notice that the only two proper subgraphs of $\Gamma$ that have four vertices that satisfy the three prime condition are the path with four vertices and the disconnected graph that consists of two paths each with two vertices.  By \cite{[2]} and Theorem 14 (c) of \cite{[3]}, these two graphs cannot occur, and so, this result follows.
\end{proof}

Assume Hypothesis 1, and let $N \lhd G$ with $\rho(N) = \rho(G)$ (or $\rho(G/N) = \rho(G)$), then $\Delta(N) = \Delta(G)$ (or $\Delta(G/N)=\Delta(G)$, respectively).  We will make strong use of this fact in the following proofs without explanation.

Throughout this section, we assume that $G$ has a normal nonabelian
Sylow $p$-subgroup $P$ for some prime $p$.  Notice that the
Schur-Zassenhaus theorem will imply that $G$ has a Hall
$p$-complement $H$.  In other words, $G$ can be viewed as a
semi-direct product $PH$.  This first easy lemma gives a condition
that relates ${\bf C}_{P^{'}} (H)$ with degrees in $\cd G$ that are
nontrivial $p$-powers.

\begin{lemma}\label{Lemma 2.1} Let $G = PH$ be the semi-direct product of $H$ acting on $P$, where $P$ is the normal Sylow $p$-subgroup of $G$ and $H$ is a $p$-complement of $P$ in $G$. Then the following are equivalent:

(1) ${\bf C}_{P^{'}} (H) \neq 1$.

(2) There exists some character $\chi \in \Irr (G)$ such that $\chi
(1)$ is a nontrivial $p$-power.
\end{lemma}

\begin{proof}  Suppose ${\bf C}_{P^{'}} (H) \neq 1$.  By the
Glauberman-Isaacs correspondence, $H$ fixes some nontrivial
irreducible character of $P^{'}$, say $\xi$.  Using Theorem 13.28 of
\cite{[4]}, $\xi^{P}$ has an $H$-invariant irreducible constituent,
say $\theta$.  By Theorem 8.15 of \cite{[4]}, $\theta$ extends to
$G$, which implies the existence of a character $\chi \in \Irr (G)$
where $\chi (1) = \theta (1)$ is a nontrivial $p$-power.

To prove the converse, suppose we have a character $\chi \in \Irr
(G)$ where $\chi (1)$ a nontrivial $p$-power.  Let $\theta =
\chi_{P}$, then $\theta \in \NL(P)$ and is $H$-invariant.  By
Theorem 13.27 of [4],\ $\theta_{P^{'}}$ has a nonprincipal
$H$-invariant irreducible constituent, say $\xi$.  By the
Glauberman-Isaacs correspondence, $\Irr ({\bf C}_{P'} (H))$ contains
a nonprincipal character, which implies ${\bf C}_{P^{'}}(H) \neq 1$.
\end{proof}

The remainder of this section will be devoted to groups that satisfy the following hypothesis.

\medskip
{\bf Hypothesis 2.} Assume Hypothesis 1, and suppose $G = PH$ is the semi-direct product where the $p'$-group $H$ acts by automorphisms on the nonabelian $p$-subgroup $P$.

\medskip
Assume Hypothesis 2.  Based on Lemma \ref{Lemma 2.1}, there are two
cases to consider, namely: (1) ${\bf C}_{P^{'}} (H) \neq 1$ and (2)
${\bf C}_{P^{'}} (H) = 1$.  The goal of the next two lemmas is to
prove that (2) cannot occur.

\begin{lemma} \label{Lemma 2.2} Assume Hypothesis 2, and suppose ${\bf C}_{P^{'}} (H) = 1$, then $A \leq {\bf Z}(G)$.  In addition, all Sylow subgroups of $H$ are abelian and the Sylow $q$-subgroup of $H$ is central in $H$.
\end{lemma}

Recall that $A$ is the normal, abelian Hall $\rho$-complement of $G$.  Thus, Lemma \ref{Lemma 2.2} says that $A$ is a direct factor of $G$.  In particular, we may assume in this case that $A = 1$. If we take $A = 1$, then $H = Q \times (RS)$ where $Q$ is the Sylow $q$-subgroup, and $RS$ is a Hall $\{r, s \}$-subgroup.

\begin{proof}  Since ${\bf C}_{P^{'}}(H) = 1$, there is no character in $\NL(P)$ extending to $G$ (this is Lemma \ref{Lemma 2.1}).  By Corollary 8.16 of \cite{[4]}, this implies that there is no character in $\NL(P)$ which is invariant under $H$.  By the structure of $\Gamma$, there exists a character $\chi_{r} \in \Irr (G)$ with $\chi_{r} (1) = p^{a}r^{b}$, where $a, b $ are positive integers.  Let $\theta \in \Irr (P)$ be a constituent of $\chi_P$.  Observe that $\theta (1)_{p} = \chi_r (1)_{p}$, so $\theta \in \NL(P)$.  We know that $\theta$ is not $G$-invariant, so $I_G (\theta) < G$.  Since $|G:I_G (\theta)| = |H:I_H (\theta)|$ divides $\chi_r (1)$, we deduce that $|G:I_G (\theta)|$ is a power of $r$ and $I_{H} (\theta)$ contains a Hall $\{q, s \}\cup \rho^{'}$-subgroup of $G$, say $B$.

Now, every degree in $\cd {G \mid \theta}$ is divisible by $\theta
(1) |G:I_G (\theta)|$.  Hence, every degree in $\cd {G \mid \theta}$
has the form $p^\alpha r^\beta$.  By Clifford's theorem, $q$ and $s$
will not divide any degrees in $\cd {I_G (\theta) \mid \theta}$.  By
Corollary 8.16 of \cite{[4]}, $\theta$ extends to $I_G (\theta)$. We
now apply Gallagher's theorem to see that $q$ and $s$ divide no
degree in $\cd {I_G (\theta)/P} = \cd {I_H (\theta)}$.  In light of
the It\^o's theorem, $B$ is abelian and normal in $I_{H}(\theta)$.
We conclude that $[A, Q] = [A, S] = [Q, S] = 1$, where $Q$ and $S$
are some Sylow $q$- and Sylow $s$-subgroups of $H$, respectively.

Similarly, if we take $\chi_{s} \in \Irr(G)$ with $\chi_{s} (1) =
p^{c}s^{d}$, where $c$ and $d$ are positive integers, then we obtain
$[A, R] = [Q, R] = 1$, where $R$ is some Sylow $r$-subgroup of $H$.
Note that $H = AQRS$ and so $A$ and $Q$ are both central in $H$.  It
is clear that $[A, P] = 1$, and we conclude that $A \leq {\bf Z}
(G)$.
\end{proof}

We now show that ${\bf C}_{P'} (H) = 1$ cannot occur.

\begin{lemma} \label{Proposition 2.3}
Assume Hypothesis 2, then ${\bf C}_{P^{'}} (H) \neq 1$.
\end{lemma}

\begin{proof}
Suppose ${\bf C}_{P^{'}} (H) = 1$, we will find a contradiction and
so complete the proof.  {}From Lemma \ref{Lemma 2.2}, it follows
that $H^{'} \leq RS$.  Let $K = P(RS)$, it is clear that $K \lhd G$.
Consider the factor group $K/P^{'} \cong \left(P/P^{'}\right) (RS)$.
It is not hard to show that $\Delta(K/P^{'})$ has two connected
components $\{r\}$ and $\{s\}$ and by the Main Theorem of
\cite{[5]}, we get that $K/P^{'}$ is of Type 4, and so we may assume
that $R \lhd RS$.  Notice that this implies that $PR$ is normal in
$K$, and hence, $PR$ is normal in $G$.

We now prove that ${\bf C}_{P^{'}} (R)$ equals either $1$ or
$P^{'}$. Note that $Q \lhd H$, so that $PQ \lhd G$ and $PQ$ is
disconnected of Type 1.  In particular, $P^{'} \leq {\bf Z}(P)$ and
$P' \le {\bf C}_P (Q)$. If ${\bf C}_{P^{'}} (R) < P^{'}$, then by
Fitting's lemma, $P^{'} = [P^{'}, R] \times {\bf C}_{P^{'}} (R)$,
where $[P^{'}, R] \neq 1$.  Now, there exists a nonprincipal
character $\xi \in \Irr ([P^{'}, R])$, with $I_{R} (\xi) < R$.  This
implies that $r$ divides every degree in $\Irr (PR \mid \xi)$.  It
follows that $r$ divides every degree in $\Irr (K \mid \xi)$, and so
$|K:I_K (\xi)| = |RS:I_{RS} (\xi)|$ is a power of $r$.

On the one hand, using Theorem 13.28 of \cite{[4]}, there exists a
character $\theta \in \Irr (P \mid \xi)$ with $I_{RS} (\theta) \geq
I_{RS}(\xi)$.  On the other hand, with $P^{'} \leq {\bf Z}(P)$, we
have $\theta_{[P',R]} = \theta (1) \xi$, and it follows that $I_{RS}
(\theta) \leq I_{RS} (\xi)$.  We conclude $I_{RS} (\theta)= I_{RS}
(\xi)$.

Applying Corollary 8.16 of [4], $\theta$ extends to $I_G (\theta)$.
Observe that $s$ divides no degree in $\cd {I_H (\theta)}$.  By
Gallagher's theorem, $s$ divides no degree in $\cd {I_G (\theta)/P}
=\cd{I_H(\theta)} = \cd {I_{RS} (\theta)}$.  Applying It\^o's
theorem, $I_{RS} (\theta)$ contains a unique Sylow $s$-subgroup, say
$S$.

For every character $\zeta \in \Irr ({\bf C}_{P^{'}} (R))$, consider
$\xi \times \zeta \in \Irr(P^{'})$.  Observe that $I_{RS} (\xi
\times \zeta) = I_{RS} (\xi) \cap I_{RS}(\zeta)$.  Since $r$ divides
every degree in $\cd {K \mid \xi \times \zeta}$, we see that $s$
does not divide $|H:I_H (\xi \times \zeta)| = |RS:I_{RS} (\xi \times
\zeta)|$.  Thus, $I_{RS} (\xi \times \zeta)$ contains the unique
Sylow $s$-subgroup $S$ in $I_{RS} (\theta)$.  It follows that $S$
centralizes every character in $\Irr ({\bf C}_{P'} (R))$, and so,
$S$ centralizes ${\bf C}_{P'} (R)$.  Recall that $Q$ centralizes
$P'$, and obviously, $R$ centralizes ${\bf C}_{P'} (R)$.  We
conclude that $H$ centralizes ${\bf C}_{P'} (R)$.  By Hypothesis 2,
${\bf C}_{P^{'}} (H) = 1$, this forces ${\bf C}_{P^{'}}(R) = 1$.

We now work to obtain the final contradiction.  Suppose first we
have ${\bf C}_{P^{'}} (R) = 1$.  Then $P^{'}=[P^{'},\ R]$, and no
nonprincipal character in $\Irr (P^{'})$ is invariant in $R$.  If
$\theta \in \NL(P)$, then $\theta_{P'} = \theta (1) \xi$ for some
nonprincipal $\xi \in \Irr (P')$.  So $I_{R} (\theta) \leq I_{R}
(\xi) < R$.  This implies that every degree in $\cd G$ that is
divisible by $p$ is also divisible by $r$, and so, $\cd G$ will have
no degree of the form $p^{a}s^{b}$, where $a,\ b$ are positive
integers, a contradiction.  Therefore, ${\bf C}_{P^{'}} (R) =
P^{'}$.

Because ${\bf C}_{P^{'}} (H) = 1$, it must be ${\bf C}_{P^{'}} (S) =
1$, for every $S \in {\rm Syl}_{s} (H)$.  It follows if $\xi \in
\Irr(P^{'})$ is nonprincipal, then $I_{RS}(\xi)$ contains no Sylow
$s$-subgroup of $RS$, and so, for every character $\theta \in
\NL(P)$, $s \mid |RS:I_{RS}(\theta)|$.  It now follows that every
degree in $\cd G$ that is divisible by $p$ is also divisible by $s$,
which again contradicts the structure of $\Delta(G)$.
\end{proof}

Now, we also assume Hypothesis 2, and we know that case 1 happens.
It is clear that $H$ is nonabelian.   By Lemma \ref{Lemma 2.1},
there exists a character $\theta \in \NL (P)$, such that $\theta$ is
extendible to $G$.  Using Gallagher's theorem, $\cd {G \mid \theta}
= \{\theta (1) b \mid b \in \cd {G/P} \}$.  It follows that $\rho
(G/P) = \rho (H) \subseteq \{ r, s \}$ and in particular that $q$
divides no degree in $\cd H$.  {}From It\^o's theorem, we see that
$H$ has an abelian normal Sylow $q$-subgroup, say $Q$.  Notice that
either $|\rho (H)| = 1$ or $\rho (H) = \{ r, s \}$.  We consider two
cases.  The first case is $|\rho (H)| = 1$.  Without loss of
generality, we may assume $\rho (H) = \{ s \}$.

\begin{lemma} \label{Proposition 2.4}
Assume Hypothesis 2 and $\rho (H) = \{ s \}$.  Then $G = H_1 \times H_2 \times A$ where $H_1$ and $H_2 = {\bf O}^{\{r,s\}'} (G)$ are characteristic subgroups of $G$ and $H_1$ is disconnected of Type 1 with $\rho (H_1) = \{ p, q \}$ and $H_2$ is disconnected of Type 4 with $\rho (H_2) = \{ r, s \}$.
\end{lemma}

\begin{proof}
By It\^o's theorem, $H$ has an abelian, normal Hall $s^{'}$-subgroup, and so, if we write $R$ for the Sylow $r$-subgroup and $Q$ for the Sylow $q$-subgroup, the Hall $s^{'}$-subgroup is $A \times R \times Q$.  This yields $H = (A \times R \times Q)S$ where $S$ is a Sylow $s$-subgroup of $H$.  We now have $PR \lhd G$.

We know there exists a character $\chi \in \Irr (G)$ with $\chi(1) = p^{a}r^{b}$, where $a,\ b$ are positive integers.  Observe that $\chi_{PR}$ is irreducible by Corollary 11.29 of \cite{[4]}.    Applying Gallagher's theorem, the only possible prime divisors of characters in $\cd {G/PR}$ are $p$ and $r$.  Since $p$ and $r$ do not divide $|G:PR|$, we deduce that $G/PR \cong H/R$ is abelian.  That is $H^{'} \leq R$.  We see that $[H,AQ] \le R \cap AQ = 1$, so $A$ and $Q$ are central in $H$.  This implies that $H = A \times Q \times RS$.  Since $A$ centralizes $P$, $A$ is central in $G$.  Also, since $S$ is isomorphic to a subgroup of $H/R$, we see $S$ is abelian.

Let $M = PQ$, and let $K = P(RS)$.  We note that $K$ and $M$ are
normal Hall subgroups of $G$.  It is not difficult to see that $\rho
(M) = \{ p, q \}$ and $\rho (K/P') = \{ r, s \}$. It follows that
$M$ and $K/P'$ are disconnected groups.  Obviously, $M$ is of Type
1, and recall that this implies that $P' \le {\bf Z}(P)$.  Since all
the Sylow subgroups of $K/P'$ are abelian, $K/P'$ is of Type 4.

By Lemma \ref{Proposition 2.3}, $1 \neq {\bf C}_{P^{'}} (H) \leq
{\bf C}_{P^{'}} (R)$.  Let $P_{0} = {\bf C}_{P^{'}} (R)$, so that
$P_{0} \neq 1$.  Suppose that $P_{0} < P^{'}$.  Next, we show that
there will be a contradiction and so prove that $P^{'} = {\bf
C}_{P^{'}} (R)$.

Since $P_0 \le P'$, we have that $P_0$ is central (and hence normal) in $P$.  We know that $Q$ centralizes $P'$, so $Q$ will normalize $P_0$.  Obviously, $R$ will normalize $P_0$.  Since $S$ normalizes $R$ and $P$, also $S$ will normalize $P_0$, and hence $P_0$ is normal in $G$.  We know that $\rho (G/P') = \{ q, r, s \}$.  Since $P_0 < P'$, we conclude that $\rho (G/P_0) = \rho (G)$, and we have seen that this implies that $\Delta (G/P_0) = \Delta (G)$.  In particular, it follows that $\Delta (K/P_0)$ has $p$ adjacent to both $r$ and $s$.

By Fitting's lemma, ${\bf C}_{P^{'}/P_{0}} (R) = 1$, so that for
every nonprincipal character $\lambda \in \Irr (P^{'}/P_{0})$, we
have that $r$ divides $|RS:I_{RS} (\lambda)|$.  Observe that
$P^{'}/P_{0} \leq {\bf Z}(P/P_{0})$, so that $r \mid |RS:I_{RS}
(\theta)|$, for every character $\theta \in \Irr (P/P_{0} \mid
\lambda)$.  And so, $r \mid \chi(1)$, for $\chi \in \Irr (G \mid
\lambda)$.  This shows that there is no character in $\Irr
(K/P_{0})$ with degree $p^{a}s^{b}$, where $a, b$ are positive
integers, a contradiction.  Thus, $R$ centralizes $P'$.  A similar
argument shows that $S$ centralizes $P'$.

Define $C = {\bf C}_P (R)$, and we just showed that $P' \le C$.
Since $R$ acts coprimely, we have $C/P' = {\bf C}_{P/P'} (R)$.  Set
$P_1 = [P,R] P'$.  By Fitting's lemma, $P/P' = P_1/P' \times C/P'$.
Since $K/P'$ is disconnected of Type 4, $P_1/P'$ is irreducible
under the action of $R$.

We now apply Lemma \ref{not square} in $K/P'$ to see that $P_1/P'
\cong P/C$ has order that is not a square.  We know that $P' \le
{\bf Z}(P)$ and $P'$ is centralized by $R$.  Also, $R$ acts
irreducibly on $P_1/P'$.  Applying Problem 6.12 of \cite {[4]}, we
see that every character in $\Irr (P')$ either extends to $P_1$ or
is fully ramified with respect to $P_1/P'$.  Since $|P_1:P'|$ is not
a square, we conclude that every character in $P'$ extends to $P_1$.
This implies that $P_1$ is abelian.  We now apply Fitting's lemma to
see that $P_1 = P' \times [P,R]$.  In particular, we have $P = C
\times [P,R]$.

Since $Q$ normalizes $P$ and $R$, we know that $Q$ normalizes $C$ and $[P,R]$.  Observe that if $Q$ acts nontrivially on $[P,R]$, then there will be some character in $\Irr (G)$ whose degree is divisible by $pq$.  Thus, $Q$ centralizes $[P,R]$.  This implies that $PQ = CQ \times [P,R]$.  If we write $H_1 = CQ$, then $\rho (CQ) = \rho (PQ) = \{ p, q \}$, and $CQ$ will be disconnected of Type 1.

We know that $R$ centralizes $C$.  Since $K/P'$ is disconnected of Type 4, it follows that $S$ centralizes $C/P'$.  We have already seen that $S$ centralizes $P'$, so this implies that $S$ centralizes $C$.  Hence, $K = PRS = C \times [P,R]RS$.
Letting $H_2 = [P,R]RS$, it is not difficult to see that $K/P' \cong C/P' \times H_2 P'/P'$ where $H_2 P'/P' \cong H_2$.  It follows that $\rho (H_2) = \{ r, s \}$ and $H_2$ is disconnected of Type 4.  Now, $H_1$ centralizes $H_2$, so $G = H_1 \times H_2 \times A$.

We now have that $(RS)^G \le H_2$.  Also, $[P,R]$ is normalized by
$H_1$, $H_2$, and $A$, so $[P,R]$ is normal in $G$.  Since it is
irreducible under the action of $R$, we see that $[P,R]$ is minimal
normal in $G$.  Hence, either $[P,R] \le (RS)^G$ or $[P,R] \cap
(RS)^G = 1$.  If $[P,R] \cap (RS)^G = 1$, then $R$ would centralize
$[P,R]$, a contradiction.  Therefore, $[P,R] \le (RS)^G$, and so,
$(RS)^G = H_2$.  This implies that $H_2$ is characteristic in $G$.
Observe that $H_1 = {\bf C}_{PQ} (H_2)$.  Since $PQ$ and $H_2$ are
characteristic in $G$, we conclude that $H_1$ is characteristic in
$G$.
\end{proof}

We now consider the second case: $\rho (H) = \{r, s\}$.

\begin{lemma} \label{Proposition 2.5}
Assume Hypothesis 2, and $\rho (H) = \{r, s\}$. Then $G = M \times
N$ where $M$ and $N = {\bf O}^{\{r,s\}'} (G)$ are characteristic
subgroups so that $\rho (M) = \{ p, q \}$ and $M$ is disconnected of
Type 1 and $P \le M$ and $\rho (N) = \{ r, s \}$ and $N$ is
disconnected of any type except Type 6.  Furthermore, if $N$ is not
of Type 4, then we may assume $M$ and $N$ are Hall subgroups.
\end{lemma}

\begin{proof}
We have that $\rho (H) = \{ r, s\}$.  Since $H \cong G/P$, we
see that $\Delta (H)$ is a subgraph of $\Delta (G)$.  In particular, $\Delta (H)$ is not connected.  Notice that groups of Type 6  have at least three primes, so $H$ cannot be of Type 6.

{}From It\^o's Theorem, $H$ has an abelian normal Sylow
$q$-subgroup, say $Q$.  It follows that $PQ \lhd G$. We note that $\Delta (PQ)$ has two connected components $\{ p \}$ and $\{ q \}$.  Since $h(PQ)=2$, $PQ$ is disconnected of Type 1.

Suppose for now that $H$ is disconnected and does not have Type 4.
The Fitting subgroups of disconnected groups of Types 1, 2, 3, and 5
all have the property that they are the direct product of a
$t$-subgroup and a central subgroup where $t \in \rho(H)$.  Thus,
$QA = Q \times A \leq {\bf Z}(H)$.

Since $H$ has a central Hall $\{r, s\}$-complement, $H$ has a
normal Hall $\{r, s \}$-subgroup, say $H_{1}$.  We have $H = AQ \times H_{1}$.  It is clear that $H_{1}$ is disconnected of the same Type as $H$.

Next, we consider the group $G_{1} = PH_{1}$. Observe that $G_1$ is
normal in $G$.  It is not difficult to see that $\Delta
(G_{1}/P^{'})$ has two connected components $\{ r \}$ and $\{ s \}$.
Notice that $G_1/P'$ has $H_1$ as a quotient.  A disconnected group
of Type 4 does not have any quotients of other disconnected types,
so $G_1/P'$ cannot be disconnected of Type 4 and since $\rho
(G_1/P')$, has only two primes, it is not disconnected of Type 6.
Because $p$ is not in $\rho (G_1/P')$, we have $P/P^{'} \leq {\bf
Z}(G_{1}/P^{'})$.  This implies that $H_1$ centralizes $P/P'$.
Hence, $P = {\bf C}_P (H_1) \Phi (P)$.  By the property of the
Frattini subgroup, we obtain $P = {\bf C}_P (H_1)$ and $H_1$
centralizes $P$.  We conclude that $G = PQ \times H_1 \times A$.  We
take $M = PQ$ and $N = H_1 A$.  Since $M$ and $N$ are normal Hall
subgroups, they are characteristic.

We now consider the case where $H$ is disconnected of Type 4.  Let $H = VL$ be the semi-direct product of $V$ with $L$ as given by the definition of groups of Type 4 where $t$ is the prime so that $V$ is a $t$-group.

Let $K = {\bf F}(L)$ and $Z = {\bf C}_{L} (V)$.  Since the connected
components of $\Delta (H)$ are $\{ r \}$ and $\{ s \}$, the index
$|K/Z|$ is a prime power, say $r^{a}$; and similarly, $|H:VK|$ is a
power of $s$.  Observe that $V$ is an elementary abelian $t$-group.
We note that $t \neq r$.

First, we assume $t \neq q$.  It is clear that $Q \leq {\bf Z}(H)$.
Let $H_1$ be a Hall $\{ q \}$-complement of $H$, so $H=H_{1} \times
Q$, and let $G_1 = PH_{1}$.  Observe that $H_1$ is disconnected of
Type 4.  Consider $G_1/P^{'}$. We note that $G_1 \lhd G$, so
$\Delta(G_1/P^{'})$ has two connected components $\{r \}$ and $\{s
\}$.  Thus, $G_1/P^{'}$ is of disconnected type.  Since $H_1$ is a
quotient of $G_1/P'$ and $H_1$ is of Type 4, we see that $G_1/P'$ is
Type 4, 5, or 6.  Since $\rho (G_1/P')$ is of size 2, $G_1/P'$ is
not of type 6, and since the Fitting subgroup of $G_1/P'$ is
abelian, it is not of Type 5.  Therefore, $G_{1}/P^{'}$ is of Type
4.  Finally, observe that $P/P^{'}\leq {\bf Z}(G_{1}/P^{'})$ since
we already know that the noncentral portion of the Fitting subgroup
is a $t$-group with $t \ne p$.  Now, $H_1$ centralizes $P/P^{'}$,
and as above, this implies that $H_1$ centralizes $P$.  Thus,
$G_{1}=P \times H_{1}$.  We conclude that $G = PQ \times H_{1}$.  We
take $M = PQ$ and $N = H_1$.  Since $M$ and $N$ are normal Hall
subgroups, they are characteristic.

Now, suppose $t = q$. Then $A \leq {\bf Z}(G)$.  Since $H$ is of
Type 4, $Q = V \times (Q \cap {\bf Z}(H))$.  Write $Q_{1} = V$ and
$Q_2 = Q \cap {\bf Z} (H)$.  Let $R$ be a Sylow $r$-subgroup of $H$.
Because $H$ is of Type 4, $Q_{1} = [Q, H] \geq [Q, {\bf F}(L)] = [Q,
R]$. Since $[Q, R] \lhd H$ and $Q_{1}$ is irreducible under the
action of $R$, we have $Q_{1}=[Q, R]$.  Write $G_{1} = P(Q_{1}RS)$
where $S$ is some Sylow $s$-subgroup of $H$ that normalizes $R$.
Since $A$ and $Q_2$ both normalize $G_1$, it follows that $G_{1}
\lhd G$.

Suppose that $[P, Q_{1}] = 1$.  Thus, $Q_1$ centralizes $P$.   In
particular, $Q_1$ is normal in $G_1$.  We see that
$\Delta(G_{1}/P^{'})$ has two connected components $\{ r \}$ and $\{
s \}$.  Thus, $G_1/P'$ is of disconnected type.  Arguing as in the
previous case, we conclude that $P/P^{'} \leq {\bf Z}(G_{1}/P^{'})$.
We deduce that $Q_1 RS$ will centralize $P$, and so, $G_{1} = P
\times (Q_{1}RS)$.  We conclude that $G = PQ_2 \times (Q_{1}RS)
\times A$. It is not difficult to see that $PQ_2$ is disconnected of
Type 1, and $Q_{1}RS$ is disconnected of Type 4.  Observe that $Q_1$
is a normal subgroup of $G$, and since $R$ acts irreducibly, it is
minimal normal in $G$.  It follows that $(RS)^G = Q_1 RS$.  It
follows that $N = Q_1 RS$ is a characteristic subgroup of $G$.
Observe that $PQ_2 A = {\bf C}_{PQA} (N)$ is characteristic since
$PQA$ and $N$ are characteristic in $G$.  Hence, we may take $M =
PQ_2 A$.

To prove the lemma, it suffices to show that $[P,Q_1] = 1$.   We
suppose that $[P, Q_{1}] \neq 1$, and we find a contradiction.
Observe that $\rho (G_{1}) = \rho (G)$, so that $\Delta (G_{1}) =
\Delta (G)$ is $\Gamma$.  It is clear that $PQ_{1}$ is disconnected
of Type 1, and so, $P^{'}=[P, Q_{1}]^{'} < [P, Q_{1}]$.  Now,
consider factor group $G_{1}/[P, Q_{1}]$.  We can see that $\Delta
(G_{1}/[P, Q_{1}])$ has two connected components: $\{ r \}$ and $\{
s \}$. And as in the previous cases, we conclude that $P/[P, Q_{1}]
\leq {\bf Z}(G_{1}/[P, Q_{1}])$, so that $[P, Q_{1}RS] \leq [P,
Q_{1}]$.

Write $P_{1} = [P, Q_{1}]$ and $G_{2} = P_{1} (Q_{1}RS)$.  Note that $G_{2} \lhd G_{1}$, since $(G_{1})^{'} \leq G_{2}$.  Observe that $P^{'} = [P, Q_{1}]^{'} = (P_{1})^{'} \neq 1$, so $P_{1}$ is nonabelian, and $[P_{1}, Q_{1}] = P_{1}$.  It is easy to see that $\rho(G_{2}) = \rho(G_{1}) = \rho(G)$, so that $\Delta(G_{2}) = \Delta (G_1) = \Delta(G)$.  Next, we consider the factor group $G_{2}/\Phi(P_{1})$ which is the semi-direct product of $P_{1}/\Phi(P_{1})$ acted on by $(Q_{1}RS)$.  By Fitting's Lemma, $P_{1}/\Phi(P_{1}) = [P_{1}/\Phi(P_{1}), Q_{1}]$.  Thus, no nonprincipal character in $\Irr (P_{1}/ \Phi(P_{1}))$ is invariant under the action of $Q_1$.

By Maschke's theorem, $P_{1}/\Phi(P_{1})$ is completely reducible
under the action of $Q_{1}$.  Choose $P_2$ with $\Phi (P) \le P_2 <
P_1$ so that $P_1/P_2$ is irreducible under the action of $Q_{1}$.
Let $Q_{3}$ be the kernel of the action of $Q_{1}$ on $P_1/P_2$.
Thus, $P_1/P_2$ is a faithful, irreducible module for $Q_{1}/Q_{3}$,
and so, $Q_{1}/Q_{3}$ is cyclic.  As $Q_{1}$ is an elementary
abelian $q$-group, $Q_{1}/Q_{3}$ is cyclic of order $q$.

Consider a nonprincipal character, $\lambda \in \Irr(P_1/P_2)
\subseteq \Irr(P_{1}/\Phi(P_{1}))$.  We know that $I_{Q_{1}}
(\lambda) < Q_{1}$, and so, $I_{Q_{1}RS} (\lambda) < Q_{1}RS$. Since
$Q_3$ is the kernel of the action of $Q_1$ on $P_1/P_2$, we have
$Q_3 \leq I_{Q_{1}}(\lambda)$.  Since $|Q_1:Q_3| = q$ and $I_{Q_1}
(\lambda) < Q_1$, we deduce that $Q_3 = I_{Q_1} (\lambda)$.

If $I_{Q_{1}RS} (\lambda)$ contains a full Sylow $r$-subgroup, then
we may assume by conjugating $\lambda$ that it contains $R$.  Hence,
$R$ will normalize $Q_3$, and this contradicts the fact that $Q_{1}$
is irreducible under the action of $R$.  Thus, $I_{Q_1 RS}
(\lambda)$ does not contain a full Sylow $r$-subgroup.  In
particular, $r$ divides $|G_2:I_{G_2} (\lambda)|$ (note that
$I_{G_2} (\lambda) = P_1 I_{Q_1 RS} (\lambda)$).  We see that
$\lambda$ extends to $I_{G_2} (\lambda)$.  Since $r$ and $s$ are not
adjacent in $\Delta (G_2)$, we conclude that $s$ does not divide
$|G_2:I_{G_2} (\lambda)|$ and $s$ divides no degree in $\cd {I_{G_2}
(\lambda) \mid \lambda}$.  Thus, $I_{Q_1 RS} (\lambda)$ contains a
full Sylow $s$-subgroup, and by conjugating $\lambda$, we may assume
$S \le I_{Q_1 RS} (\lambda)$.  Applying Gallagher's theorem, $s$
divides no degree in $I_{G_2} (\lambda)/P_1 \cong I_{Q_1 RS}
(\lambda)$.  By It\^o's theorem, $S$ is normal in $I_{Q_1 RS}
(\lambda)$.  It follows that $S$ centralizes $Q_3$.  Since $S$ does
not centralize $Q_1$, we have ${\bf C}_{Q_1} (S) = Q_3$.  Recall
that we may view $S$ as acting like a Galois group on a field whose
additive group is isomorphic to $Q_{1}$, and $Q_3$ will be the fixed
field under $S$.  Thus, if $|Q_{1}|=q^{a}$ and $|Q_{3}|=q^{b}$, then
$b \mid a$, which contradicts with $a=b+1$. Therefore, $[P, Q_{1}] =
1$, as desired.
\end{proof}

Since Lemmas \ref{Proposition 2.4} and \ref{Proposition 2.5} cover the only possibilities for Hypothesis 2, we can combine them to obtain.

\begin{theorem} \label{nonab p}
Suppose Hypothesis 2.  Then $G = H_1 \times H_2$ where $H_1$ and $H_2$ are characteristic subgroups of $G$ and disconnected groups where $\rho (H_1) = \{ p, q \}$ is of Type 1 and $\rho (H_2) = \{ r, s \}$ is of any type except Type 6.  Furthermore, if $H_2$ is not of Type 4, then we may assume $H_1$ and $H_2$ are Hall subgroups.
\end{theorem}

We also obtain a corollary about $G$ with $h(G)=2$.

\begin{corollary} \label{Lemma 3.2}
Suppose Hypothesis 1, and suppose $h(G) = 2$.  Then $G = H_1 \times H_2$ where $H_1$ and $H_2$ are characteristic Hall subgroups of $G$ and are disconnected groups of Type 1 with $\rho (H_1) = \{ p, q \}$ and $\rho (H_2) = \{ r, s \}$.
\end{corollary}

\begin{proof}  Let $F$ be the Fitting subgroup of $G$.  We may use the discussion before theorem 19.6 of \cite{[6]} to see that $\cd G$ contains a degree $a$ that is divisible by every prime divisor of $|G:F|$.  It follows that $|G:F|$ is divisible by at most two of the primes in $\rho (G)$.  Hence, $G$ must have a nonabelian normal Sylow $p$-subgroup for some prime $p$.  We then may apply Theorem \ref{nonab p} to see that $G = H_1 \times H_2$ where $H_1$ and $H_2$ are characteristic subgroups and are disconnected groups with $\rho (H_1) = \{ p, q \}$ and $\rho (H_2) = \{ r, s \}$.  We know $H_1$ is of Type 1.  If $H_2$ is not of Type 1, then $h(G) > 2$, a contradiction, so $H_2$ is also of Type 1.
\end{proof}

We also can characterize $G$ if it has more than one nonabelian Sylow subgroup.

\begin{corollary} \label{two nonab}
Suppose Hypothesis 1, and assume $G$ has more than one nonabelian normal Sylow subgroup.  Then $h (G) = 2$.
\end{corollary}

\begin{proof}
Since $G$ has at least one nonabelian normal Sylow subgroup, we may apply Theorem \ref{nonab p} to see that $G = H_1 \times H_2$ where $H_1$ and $H_2$ are disconnected groups where $H_1$ has Type 1 and $H_2$ has any Type except Type 6.  However, Types 2, 3, 4, and 5 do not have any nonabelian normal Sylow subgroups.  Thus, the only way $G$ can have more than one nonabelian normal Sylow subgroup is if $H_2$ is also of Type 1, and in this case, $h(G) = 2$.
\end{proof}

\section{No normal nonabelian Sylow subgroups} \label{section 3}

We now consider groups where there are no normal nonabelian Sylow subgroups.  With this in mind, we make the following hypothesis that we study throughout this section.

\medskip
{\bf Hypothesis 3.} Assume Hypothesis 1, and suppose that for every prime $t$ belonging to $\rho(G)$, $G$ has no normal Sylow $t$-subgroup.  Let $F$ and $E/F$ be the Fitting subgroups of $G$ and $G/F$.

\medskip
We now consider a series of lemmas that study Hypothesis 3.  The first one shows that $G/\Phi (G)$ has the same character degree graph as $G$.

\begin{lemma} \label{Lemma 3.1} Assume Hypothesis 3,
then $\Delta (G) = \Delta (G/\Phi (G))$, where $\Phi (G)$ is the Frattini subgroup of $G$.
\end{lemma}

\begin{proof}
It is clear that $\rho (G/\Phi(G)) \subseteq \rho(G)$.  Suppose
there exists a prime $p \in \rho(G) \setminus \rho(G/\Phi(G))$.  By
It\^o's theorem, $G/\Phi(G)$ has an abelian, normal Sylow
$p$-subgroup, say $P\Phi(G)/\Phi(G)$, where $P \in Syl_{p}(G)$.
Using the Frattini Argument, $G = {\bf N}_G (P) (P \Phi (G)) = {\bf
N}_G (P) \Phi (G)$, so $G = {\bf N}_G (P)$.  We obtain $P \lhd G$,
contradicting Hypothesis 3.  Now, we get $\rho(G) = \rho(G/\Phi(G))$
and so $\Delta (G) = \Delta (G/\Phi(G))$.
\end{proof}

By using Lemma \ref{Lemma 3.1} to determine the structure of $G$, we discuss groups $G$ with a trivial Frattini subgroup first.  Then we use the results of $G/\Phi(G)$ to get the structure of $G$.

If $h(G) = 2$, then we showed in Corollary \ref{Lemma 3.2}, that the desired conclusion holds.  Using Theorem B of \cite{[1]}, we know that $h(G) \le 4$ for groups satisfying Hypothesis 1, thus, we just need to consider $G$ with $h(G) = 3$ and $4$.  In particular, we have $E < G$. The next lemma specifically considers $h(G) = 3$.

\begin{lemma} \label{Lemma 3.3}
Assume Hypothesis 3, and suppose that $h(G) = 3$.  Then the following hold:
\begin{enumerate}
\item $|E/F|$ has two distinct prime divisors say $\{p, r\}$.
\item $E/F$ is the Hall $\{p, r\}$-subgroup of $G/F$.
\item $G/F$ has an abelian Hall $\{q, s\}$-subgroup.
\end{enumerate}
\end{lemma}

\begin{proof}
We know that $\pi(G/F) = \rho(G) = \{p, q, r, s\}$.  Note that $\pi(G/F) = \pi(E/F) \cup \pi(G/E)$.  Using the discussion before Theorem 19.6 of \cite{[6]}, there exists a character $\chi \in \Irr (G)$ whose degree is divisible by every prime in $\pi(E/F)$; and a character $\psi \in \Irr (G)$ whose degree is divisible by every prime in $\pi(G/E)$.  Looking at $\Delta(G)$, we deduce that each of $|E:F|$ and $|G:E|$ are divisible by 2 primes.  From this, one can conclude that $E/F$ is a Hall subgroup of $G/F$, and without loss of generality, we may assume that the primes dividing $|E:F|$ are $p$ and $r$.

Choose a character $\varphi \in \Irr (G)$ with $\varphi(1) =
p^{a}r^{b}$, where $a$ and $b$ are positive integers. Let $\theta$
be an irreducible constituent of $\varphi_{E}$.  Observe that
$\theta$  extends to $\varphi$ on $G$. By Gallagher's theorem, we
obtain $\varphi (1) a \in \cd G$ for every degree $a \in \cd {G/E}$.
However, since $a$ divides $|G:E|$, we can use the structure of
$\Delta(G)$, to see that $a = 1$.  Thus, $G/E$ is abelian, which
implies $G/F$ has an abelian Hall $\{q, s\}$-subgroup.
\end{proof}

We now consider Hypothesis 3 with the additional condition that $\Phi (G) = 1$.  This next lemma is really just a restatement of Gash\"utz's theorem.

\begin{lemma} \label{comp}
Assume Hypothesis 3 and suppose $\Phi (G) = 1$.  Then there exists a complement $L$ for $F$ in $G$ and $F = [F,E] \times Z$ where $E$ centralizes $Z$ and $E \cap L = {\bf F}(L)$ acts faithfully on $[F,E]$.
\end{lemma}

\begin{proof}
Note that Gash\"utz's theorem (see Hilfsatz III.4.4 of \cite{[7]})
implies the existence of a complement $L$ for $F$ in $G$.  It is not
difficult to see that $L$ acts faithfully on $F$, and so, $E \cap L=
F(L)$ acts faithfully on $[F, E]$.  This implies that $L$ must act
faithfully on $[F, E]$.  Since $\Phi(G)=1$, Gash\"utz's theorem
(Satz III.4.5 of \cite{[7]}) also tells us that $F$ is completely
reducible under the action of $L$.  In particular, $F = [F,E] \times
Z$ where $Z$ is normal in $G$.  Observe that $[Z,E] \le Z \cap [F,E]
= 1$, so $E$ centralizes $Z$.
\end{proof}

We study the case when $\Phi (G) = 1$ much more closely.  Hence, we make the following hypothesis.

\medskip
{\bf Hypothesis 4.} Assume Hypothesis 3, and suppose $\Phi(G) = 1$.
Let $L$ be a complement for $F$ in $G$.
\medskip

The following lemma turns out to be key to our work.

\begin{lemma} \label{stabs}
Assume Hypothesis 4.  Then no prime divides $|L:I_L (\lambda)|$ for
all nonprincipal characters $\lambda \in \Irr ([F,E])$.
\end{lemma}

\begin{proof}
Now, consider a nonprincipal character $\lambda \in \Irr([E, F])$.
Observe that the stabilizer of $\lambda$ in $G$ is $F I_{L}
(\lambda)$.  Furthermore, we know that $I_{E \cap L}(\lambda) < E
\cap L $. It follows that some prime divisor of $|E:F|$ divides
$|L:I_{L}(\lambda)|$.

Suppose there exists a fixed prime $t \in \pi (|E:F|)$, say $t = p$,
such that for every nonprincipal character $\lambda \in \Irr([E,
F])$, we have $p \mid |L:I_{L} (\lambda)|$.  We note that $q$ is not
adjacent to $p$ in $\Delta(G)$, and every degree in $\cd {G \mid
\lambda}$ is divisible by $|L:I_{L} (\lambda)|$.  We see that
$I_{L}(\lambda)$ contains a full Sylow $q$-subgroup of $L$.
Furthermore, by Corollary 2 of \cite{[2]}, we know that $\lambda$
extends to $FI_{L} (\lambda)$.  We see that $q$ divides no degree in
$\cd {FI_{L} (\lambda) \mid \lambda}$.  By Gallagher's theorem, $q$
divides no degree in $\Irr (FI_L (\lambda)/F) = \Irr (I_{L}
(\lambda))$, and using It\^o's theorem, $I_{L} (\lambda)$ has a
normal Sylow $q$-subgroup. This is the situation of Lemma 1 of
\cite{[1]}.  From that lemma, either (1) $|[E, F]|=9$ and $L$ is
isomorphic to either ${\rm SL}_{2}(3)$ or ${\rm GL}_{2}(3)$, or (2) There is a
normal subgroup $O \leq L$ so that $O$ is abelian, $[E, F]$ is
irreducible under the action of $O$, and $L \le \Gamma ([E,F])$, in
particular, we have $h(L)\leq 2$ and $h(G)\leq 3$.

Since $L \cong G/F$, we have $\pi(L) = \pi(G/F) = \{p, q, r, s\}$,
so that (1) cannot occur.  Thus, (2) happens.  At this time, if
$h(G) = 4$, then we get a contradiction, so we assume that $h(G) =
3$.  In particular, we have $O \leq {\bf F}(L)$.

By Lemma \ref{Lemma 3.3}, ${\bf F}(L) \cong E/F$ is the Hall $\{p,
r\}$-subgroup of $G/F$.  Replacing $O$ by $O {\bf Z}({\bf F}(L))$,
we may assume that $|O|$ is divisible by $p$ and $r$.  (Note that
$O{\bf Z}({\bf F}(L))$ is abelian and $[E,F]$ is irreducible under
the action of $O{\bf Z}({\bf F}(L))$.) By Theorem 2.1 of \cite{[6]},
$O \leq \Gamma_{0}([E,\ F])$, and $O$ acts Frobeniusly on $[E, F]$.
Therefore, $[E, F]$ is a $\{ p, r \}^{'}$-group.  Also, $p$ and $r$
will divide $|L:I_L (\lambda)|$ for every nonprincipal character
$\lambda \in \Irr ([F,E])$. Applying Gallagher's theorem, this
implies that $p$ and $r$ are the only primes dividing degrees in
$\cd {I_L (\lambda)}$. By It\^o's theorem, $I_L (\lambda)$ has a
normal (unique) $\{ p, r \}$-complement. By Clifford's theorem, $p$
and $r$ are the only primes dividing $|G:I_G (\lambda)| = |L:I_L
(\lambda)|$.  Therefore, $I_L (\lambda)$ contains a unique Hall $\{
p, r \}$-complement of $L$.

Following Lemma \ref{comp}, we have $F = [F,E] \times Z$ where $E$
centralizes $Z$.  Now consider a nonprincipal character $\mu \in
\Irr (Z)$ and write $C$ for its stabilizer in $G$.  The stabilizer
of $\lambda \times \mu$ in $G$ is $C \cap T$, where $T$ is the
stabilizer of $\lambda$ in $G$.  Since $O$ acts Frobeniusly on
$[E,F]$, we have $pr \mid |G:T| \mid |G:C \cap T|$.  Since $|G:C
\cap T|$ divides degrees in $\cd G$, it is a $\{p, r\}$-number.  In
particular, $C$ contains the unique Hall $\{p, r\}$-complement found
in $T$.  Also $C$ contains $E$.  Since $|G:E|$ is not divisible by
$p$ or $r$, we conclude that $G = C$.  We conclude that every
character of $Z$ is $G$-invariant, and therefore, $Z = {\bf Z}(G)$.
Also, every degree in $\cd {G \mid F}$ is a $\{ p, r \}$-number.

From the structure of $\Delta (G)$, we know that $\cd G$ contains
$\{ p, s\}$-numbers and $\{ r, q \}$-numbers.  By the previous
paragraph, these degrees must lie in $\cd {G/F} = \cd L$.  It
follows that $\rho (L) = \rho (G)$ and so, $\Delta (L) = \Delta
(G)$.  We note that $h(L) = 2$, and $\Delta (L) = \Gamma$, it
follows from Corollary \ref{Lemma 3.2} that $L = PQ \times
RS$, where $P, Q, R, S$ are Sylow $p, q, r,
s$-subgroups of $L$, respectively. Now, let $L_{1} = PQ$,
$L_{2} = RS$, and $G_{i} = FL_{i}$, where $i = 1, 2$. Notice
that $L_1$ and $L_2$ are both disconnected of Type 1, so $P$ and
$R$ are both nonabelian.   It is clear that $G_{i}$ is normal in
$G$, $G_{1}$ contains the Hall $\{p, q\}$-subgroups of $G$, and
$G_{2}$ contains the Hall $\{r, s\}$-subgroups of $G$. And $F = {\bf
F}(G_{i})$ is abelian for $i=1, 2$.  And so $\Delta(G_{1})$ is
disconnected with two connected components $\{p\}$ and $\{q\}$;
$\Delta(G_{2})$ is disconnected with two connected components
$\{r\}$ and $\{s\}$.  Observe that $h(G_{i}) = 3$, where ${\bf
F}(G_{i})$ is abelian, and both $P$ and $R$ are nonabelian.
We deduce that $G_{i}$ is disconnected of Type 2 for both $i = 1$
and $2$, which is a contradiction since the primes in Type 2 are
$\{2 , 3\}$.
\end{proof}

We are able to apply the previous lemma to the more general setting of Hypothesis 3.

\begin{corollary} \label{Fitting}
Assume Hypothesis 3.  Then $|E:F|$ has exactly two prime divisors.
\end{corollary}

\begin{proof}
Without loss of generality, we may assume that $\Phi (G) = 1$, and
we can take $L$ to be a complement for $F$ in $G$.  By Lemma
\ref{comp}, we know that $E \cap L = {\bf F} (L)$ acts faithfully on
$[E,F]$.  If $E/F$ is a $t$-group for some prime $t$, then $t$ will
divide $|L:I_L (\lambda)|$ for every nonprincipal character
$\lambda$ which violates Lemma \ref{stabs}.  Thus, $|E:F|$ is
divisible by at least two primes.  On the other hand, using the
discussion prior to the Theorem 19.6 of \cite{[6]}, we know that
some character degree in $\cd G$ is divisible by the prime divisors
of $|E/F|$. Thus, $|E/F|$ has at most two distinct prime divisors.
\end{proof}

We return to Hypothesis 4, and we consider the Sylow subgroups of $L$.  Notice that the symmetry in the hypotheses allows us to exchange both $p$ and $r$ and $q$ and $s$ in the conclusion.

\begin{lemma} \label{sylows}
Assume Hypothesis 4, and $\pi (E:F) = \{ p, r \}$.  Then there
exists $\lambda \in \Irr ([E,F])$ such that $p$ divides $|L:I_{L}
(\lambda)|$ and $r$ does not divide $|L:I_{L} (\lambda)|$.  In
particular, $[{\bf O}_{r} (L), Q] = 1$, where $Q$ is a Sylow
$q$-subgroup of $L$.
\end{lemma}

\begin{proof}
By Corollary \ref{Fitting}, $|E:F|$ is divisible by two primes, and we may assume that $|E:F|$ is divisible by $p$ and $r$.  Since ${\bf F}(L)$ acts faithfully on $F$, we know that either $p$ or $r$ divides $|L:I_L (\lambda)|$ for every nonprincipal character $\lambda \in \Irr ([F,E])$.  By Lemma \ref{stabs}, it follows that there exists characters $\lambda \in \Irr([E, F])$ so that $p$ divides $|L:I_{L} (\lambda)|$ and $r$ does not divide $|L:I_{L} (\lambda)|$. 

Since $\lambda$ extends to $I_G (\lambda) = F I_L (\lambda)$, we
know that $I_L (\lambda)$ contains a Sylow $q$-subgroup of $L$ say
$Q$ as a normal abelian subgroup.  Also, $I_L (\lambda)$ contains
a Sylow $r$-subgroup of $L$, say $R$. If $s$ divides $|L:I_{L}
(\lambda)|$, then $R$ is a normal abelian subgroup of $I_L
(\lambda)$. Otherwise, $s$ does not divide $|L:I_{L} (\lambda)|$,
and so, $I_{L} (\lambda)$ contains a Hall $\{r, s\}$-subgroup of
$L$.  In either case, ${\bf O}_r (L)$ is contained in $I_L
(\lambda)$, so we have $[{\bf O}_{r}(L), Q] = 1$.
\end{proof}

We also get information about a Hall $\{ q, s \}$-subgroup of $L$.

\begin{lemma} \label{Hall}
Assume Hypothesis 4.  If $\pi (E:F) = \{ p, r \}$.  Then $L$ has an abelian Hall $\{q , s \}$-subgroup.
\end{lemma}

\begin{proof}
Let $Y/F = {\bf Z}(E/F)$. Since ${\bf F}(Y) = F$ and $Y/F$ is
abelian, we use Theorem 18.1 of \cite{[6]} to see that there exists
a character $\xi \in \Irr (F)$ such that $\xi^{Y} \in \Irr (Y)$.
Write $\theta = \xi^{Y}$, then $\pi (\theta(1)) = \{p, r\}$.  It is
not hard to see that $pr$ divides $|L:I_{L} (\xi)|$ and so, $pr$
divides every degree in $\cd {G \mid \xi}$.  By the structure of
$\Delta (G)$, $I_{L} (\xi)$ contains a Hall $\{q, s \}$-subgroup of
$L$. Furthermore, $\xi$ extends to $F I_{L} (\xi)$.  We see that
$qs$ divides no degree in $\cd {F I_{L}(\xi) \mid \xi}$.  By
Gallagher's theorem, $qs$ divides no degree in $\cd {F I_L (\xi)/F}
= \cd {I_L (\xi)}$.  Applying It\^o's theorem, $I_{L} (\xi)$
contains an abelian, normal Hall $\{ q, s \}$-subgroup of $L$, which
implies $L$ has an abelian Hall $\{ q, s \}$-subgroup.
\end{proof}

Finally, we come to the main result about groups that satisfy Hypothesis 4.  We will apply this result to obtain the conclusion of the Main Theorem under Hypothesis 3.

\begin{theorem} \label{punchline}
Assume Hypothesis 4, and  $\pi (E:F) = \{ p, r \}$.  Let $L_{1}$ be
a Hall $\{p, q\}$-subgroup of $L$ and $L_{2}$ a Hall $\{r,
s\}$-subgroup of $L$, and write $G_{i} = F L_{i}$ for $i=1, 2$. Then
$L = L_{1} \times L_{2}$ and one of $G_{1}$ or $G_2$ is disconnected
of Type 4 and the other is disconnected of either Type 2, 3 or 4
with $\rho (G_{1}) = \{ p, q \}$ and $\rho (G_{2}) = \{ r, s \}$.
\end{theorem}

\begin{proof}
Let $Q$ be a Sylow $q$-subgroup of $L$ and $S$ be a Sylow
$s$-subgroup of $L$. And (by conjugating if necessary) we may assume that $Q$ and $S$ lie in
some Hall $\{ q, s \}$-subgroup of $L$. By Lemma \ref{Hall}, this
subgroup is abelian so $[Q, S] = 1$. If $h(G) = 3$, then ${\bf
O}_p (L)$ and ${\bf O}_r (L)$ are the Sylow $p$- and Sylow
$r$-subgroups of $L$ by Lemma \ref{Lemma 3.3}. Hence, $L_1 = {\bf
O}_p (L) Q$ and $L_2 = {\bf O}_r (L) S$.  The result $L = L_1
\times L_2$ now follows by Lemma \ref{sylows}.

It is clear that $G_{i}$ is normal in $G$ for $i = 1,2$, that $G_{1}$ contains the Hall $\{ p, q \}$-subgroups of $G$, and that
$G_{2}$ contains the Hall $\{r, s \}$-subgroups of $G$.  It follows that
$\Delta(G_{1})$ is disconnected with two connected components $\{ p
\}$ and $\{ q \}$ and $\Delta(G_{2})$ is disconnected with two
connected components $\{ r \}$ and $\{ s \}$.  By the Main theorem
of \cite{[5]}, each $G_{i}$ is of Type 2 or 4 since $h (G_{i}) = 3$
and $F = {\bf F}(G_{i})$ is abelian, for $i = 1, 2$.  We note that
$\rho (G_{1}) \cap \rho (G_{2}) = \emptyset$, so $G_{1}$ and $G_{2}$
cannot both be Type 2.  This proves the conclusion when $h(G) = 3$.

Thus, we may assume that $h(G) = 4$.  We first prove that $L$ has a
normal Sylow $t$-subgroup for some prime $t$.  We note that $L$ has
a normal Sylow $t$-subgroup if and only if $L/\Phi(L)$ has a normal
Sylow $t$-subgroup.  Since ${\bf F}(L) \cong E/F$ is a $\{ p, r
\}$-group, if $L$ has a normal Sylow $t$-subgroup, then $t \in \{ p,
r \}$.

Suppose $L$ has no normal Sylow $t$-subgroup for any prime $t$. Then
$\pi (L/{\bf F}(L)) = \pi (L/\Phi (L)) = \pi (G/F) = \rho(G)$, and
so, $\rho (L/\Phi (L)) = \rho (G)$.  It follows that $\Delta (L/\Phi
(L)) = \Delta (G)$.  Hence, $L/\Phi (L)$ satisfies Hypothesis 3.
Now, $L/\Phi (L)$ satisfies the hypotheses of the theorem with
$h(L/\Phi (L)) = 3$, and we have already proved the conclusion in
this case.  In particular, $L/{\bf F} (L) = L_1 {\bf F} (L)/{\bf F}
(L) \times L_2 {\bf F} (L)/{\bf F}(L)$, and both $L_{1} {\bf F}
(L)/\Phi (L)$ and $L_2 {\bf F} (L)/\Phi (L)$ are disconnected groups
one of Type 4 and the other of Type 2 or 4. Without loss of
generality, $L_1 {\bf F}(L)/\Phi (L)$ is of Type 4. From the
structure of disconnected groups, we know that ${\bf F}(L_1 {\bf
F}(L)/\Phi (L)) = {\bf F}(L)/\Phi (L) = V \times Z$ where $V$ is a
minimal normal subgroup of $L_1 {\bf F}(L)/\Phi (L)$ that $Q$ does
not centralize and $Z$ is central in $L_1 {\bf F}(L)/\Phi (L)$.
Since ${\bf F} (L)$ is a $\{p, r \}$-group, $V$ is an elementary
abelian $u$-subgroup for some prime $u \in \{ p, r \}$. Observe that
$V \le {\bf O}_u (L/\Phi (L))$.

We cannot have $u = r$ since $Q$ acting nontrivially on $V$ would
contradict with $[{\bf O}_{r} (L), Q] = 1$ from Lemma
\ref{sylows}.  Since $L_1$ is a Hall $\{ p, q \}$-subgroup of $L$,
we have $L_1 \Phi (L)/ \Phi (L)$ is a Hall $\{ p, q \}$-subgroup of
$L/\Phi (L)$.  If $u = p$, then $V$ is a normal $p$-subgroup of
$L/\Phi (L)$, and thus, $V$ is contained in every Hall $\{p, q
\}$-subgroup of $L/\Phi (L)$, and in particular, $V \le L_1 \Phi
(L)/ \Phi (L)$.  Thus, $L_1 \Phi (L)/\Phi (L)$ is complemented by
the central subgroup $Z$ in $L_1 {\bf F}(L)/ \Phi (L)$.  This
implies that $L_1 \Phi (L)/\Phi (L)$ is normal in $L_1 {\bf F}
(L)/\Phi (L)$.  Since $L_1 \Phi (L)/\Phi (L)$ is also a Hall
subgroup, it is characteristic in $L_1 {\bf F}(L)/\Phi (L)$, and so,
$L_1 \Phi (L)$ is normal in $L$.  This implies that $L_1$ is normal
in $L$, and $L_1 F$ is normal in $G$.  Since $F = {\bf F}(L_1 F)$ is
abelian, we conclude that $\rho (L_1 F) = \pi (L_1) = \{ p, q \}$,
and $L_1 F$ is a disconnected group.  But $L_1 F$ has Fitting height
$4$ since $L_1$ has Fitting height $3$, and so, $L_1 F$ is
disconnected of Type 3, and this implies $\{ 2, 3 \} = \{ p, q \}$.

Since $\{ r, s \} \cap \{ p, q \}$ is empty, we see that $L_2 {\bf
F} (L)/\Phi (L)$ must be disconnected of Type 4.  Repeating the
argument of the previous two paragraphs with $L_2$ in place of $L_1$
and $\{ r, s \}$ in place of $\{ p, q \}$, we conclude that $L_2 F$
is disconnected of Type 3, and $\{ r, s \} = \{ 2, 3 \}$.  This also
contradicts $\{ p, q \} \cap \{ r, s \}$ is empty, so we conclude
that $L$ has a normal Sylow $t$-subgroup.

We now know that $L$ has a normal Sylow $t$-subgroup for some $t \in \{ p, r \}$, and without loss of generality, we take $t = p$, and we write $P$ for the normal Sylow $p$-subgroup of $L$.  Without loss of generality, we may assume $L_1 = PQ$.

Suppose $P$ is nonabelian.  By Lemma \ref{sylows}, we can find
$\lambda \in \Irr ([E,F])$ where $p$ divides $|L:I_L (\lambda)|$ and
$r$ does not divide $|L:I_L (\lambda)|$.  Conjugating $\lambda$ if
necessary, $Q$ is a normal subgroup of $I_L (\lambda)$, and $L$
has a Sylow $r$-subgroup $R$ which is contained in $I_L
(\lambda)$.  Thus, $R$ normalizes $Q$.  We have already seen
that $[Q, S] = 1$.  Since $P$, $Q$, $R$, and $S$ all normalize $L_1 = PQ$, it follows that $L$ normalizes $L_1$, and so, $L_1$ is normal in $L$.  Now, $\Delta
(L_1)$ is disconnected with components $\{ p \}$ and $\{ q \}$.
Since it has Fitting height $2$, the disconnected group $L_1$ is of Type 1.

Recall that $G_1 = FL_1$, so now, $G_{1}$ is normal in $G$ and
$\Delta(G_{1})$ is disconnected with two components $\{p \}$ and
$\{q \}$.  Since $G_1$ has Fitting height 3 and ${\bf F} (G_1/F)$ is
not abelian, $G_{1}$ is disconnected of Type 2, and so $G_{1}/F
\cong SL_{2} (3)$.  Now, $p=2$, $q=3$, $P \cong Q_{8},$ and
$Q \cong Z_{3}$.  If $[L_1, R] \neq 1$, then either
$[P, R] \neq 1$ or $[Q, R] \neq 1$.  We note that
$\Aut (P) \cong S_{4}$, so if $[P, R] \neq 1$, then $r
\in \{ 2, 3 \}$, a contradiction.  On the other hand, $\Aut (Q)
\cong S_3$, so if $[Q, R] \neq 1$, then $r = 2$, again a
contradiction.  We deduce that $[L_1, R]= 1$.

Applying Lemma \ref{sylows}, we can find $\eta \in \Irr ([E,F])$ so
that $r$ divides $|L:I_L (\eta)|$ and $p$ does not divide $|L:I_L
(\eta)|$.  Conjugating $\eta$ if necessary, we may assume $S$ is a normal
subgroup of $I_L (\eta)$ and $P$ is contained in $I_L (\eta)$.
Since $P$ is normal in $L$, this implies that $[P,S] = 1$.  As
we already have $[Q,S] = 1$, we obtain $[L_1, S] =
1$.  Now, it is clear that $L = L_1 \times L_2$.
Observe that $h (L) = 3$ and $h (L_1) = 2$, and hence, $h
(L_2) = 3$.

As $G_{2} = FL_2$, we have $G_2$ is normal in $G$,
so $F = {\bf F}(G_2)$.  We see that $h (G_{2}) = 4$ and $\Delta
(G_{2})$ is disconnected with components $\{ r \}$ and $\{ s \}$.
This implies that $G_2$ is disconnected of Type 3, and thus,
$\rho(G_{2}) = \{ r, s \} = \{ 2, 3 \}$.  This is a contradiction
since $\{p, q \} \cap \{r, s \} = \emptyset$.  This proves that
$P$ is abelian.

We next show that $[Q, R] = 1$.  Suppose that $[Q, R]
\neq 1$, and we obtain a contradiction. Observe that if $s$ divides
$|L:I_L (\lambda)|$, then we have $[Q, R] = 1$.  Hence, $s$ does
not divide $|L:I_L (\lambda)|$. Let $P_{1} = P \cap I_{L}
(\lambda)$, and observe that $P_{1} \lhd I_{L} (\lambda)$ since
$P \lhd L$.  We have $I_{L} (\lambda) = (P_{1} \times Q)
RS$.  Since $Q$ is normal in $I_L (\lambda)$, if $R$
were normal, then we would have $[Q,R] = 1$, a contradiction.
Thus, $R$ is not normal in $I_L (\lambda)$.

If $S$ is normal in $I_{L}(\lambda)$, then $Q$ and $R$ both
normalize $S$, and since $P$ is in $I_L (\eta)$, we know that
$P$ will normalize $S$, and thus, $S$ is normal in $L$.  This
is a contradiction since $s$ does not divide $|{\bf F}(L)|$.  We
conclude that $S$ is not normal in $I_L (\lambda)$.  Since $P$ and
$Q$ are abelian, we see that $\rho (I_L (\lambda)) = \{ r, s \}$.
We now use the structure of $\Delta(G)$ to get that $\Delta (I_{L}
(\lambda))$ has two connected components $\{ r \}$ and $\{ s \}$.
Recall that $[Q, S] = 1$.   We claim that this implies that
$Q \leq {\bf Z}(I_{L}(\lambda))$, a contradiction.  (Observe that
the Fitting subgroup of $I_L (\lambda)$ is the direct product of a
$t$-group $V$ times a central subgroup where $S$ acts
nontrivially on $V$ since $I_L (\lambda)$ has a disconnected graph.)
We conclude that $[Q, R] = 1$.

We have now shown that $R$ and $S$ both centralize $Q$.  Thus, $C_L (Q)$ contains a Hall $\{ r, s \}$-subgroup of $L$.  Without loss of generality, we may say that it is $L_2 =
R S$.  Note that $T = PL_2$ is a Hall $\{ p, r, s
\}$-subgroup of $L$.  Since $[Q, L_2]=1$ and $P$ is normal in $L$, we observe that $T$ is normal in $L$.  Thus, $\Delta (T)$ has two
connected components $\{r \}$ and $\{s \}$.  Observe that $[P,
S] = 1$.  As before, we can conclude that $P \leq {\bf
Z}(T)$ from the properties of the Fitting subgroup of groups with
disconnected graphs.  And so $T = P \times L_2$, and hence, $L = L_1 \times L_2$.

Since $P$ is normal and abelian, $p$ is not in $\rho (L)$.   Observe that $h(L) = h(L_2)
= 3$ since $h(L_1) = 2$ and so, $h(G_{2}) = 4$, so $G_{2}$ is disconnected of Type 3
and $G_{2}/F \cong {\rm GL}_{2}(3)$.
\end{proof}

We now apply Theorem \ref{punchline} to get the result under Hypothesis 3.  Our proof breaks up into two pieces depending on whether or not $F$ is abelian.  We first handle the case when $F$ is abelian.

\begin{theorem}\label{Corollary 3.6}
Assume Hypothesis 3, and $F$ is abelian.  Then $G = H_1 \times H_2
\times Y$ where $Y$ is central, $\rho (H_1) = \{ p, q \}$ and $\rho
(H_2) = \{ r, s \}$ with at least one of $H_{1}$ and $H_{2}$ of Type
4 and the other is of Type 2, 3, or 4.  
\end{theorem}

\begin{proof}
Let $\Phi = \Phi (G)$, and let $L/\Phi$ be a complement for $F/\Phi$ in $G/\Phi$.  Following Theorem \ref{punchline}, we write $L/\Phi = L_1/\Phi \times L_2/\Phi$ where $L_1/\Phi$ is a Hall $\{ p, q \}$-subgroup of $L/\Phi$ and $L_2/\Phi$ is a Hall $\{ r, s \}$-subgroup of $L/\Phi$.  We take $G_i = FL_i$, so $G_i$ is normal in $G$.  Since $F$ is abelian, we have $\rho (G_1) = \rho (G_1/\Phi) = \{ p , q \}$ and $\rho (G_2) = \rho (G_2/\Phi) = \{ r, s \}$.  By Theorem \ref{punchline}, we assume that $G_1/\Phi$ is disconnected of Type 4, and we know that $G_2/\Phi$ is disconnected of Type 2, 3, or 4.  It is not difficult to see that $G_i$ will be of the same Type as $G_i/\Phi$.

Because each $G_i$ is disconnected of Types 2, 3, or 4, we have $F =
[G_i,F] \times Z_i$ for each $i = 1, 2$ where $Z_i ={\bf Z} (G_i)$
and $[G_i,F]$ is minimal normal in $G_i$.  Since $G_i$ is normal in
$G$, it follows that each $[G_i,F]$ is minimal normal in $G$.  We
know that $[G_i,F] \Phi > \Phi$, and $[G_i,F]$ is not contained in
$\Phi$.  By the minimality of $[G_i,F]$, we deduce that $[G_i,F]
\cap \Phi = 1$.  This implies that $[\Phi,G_i] \le [F,G_i] \cap \Phi
= 1$, so $\Phi \le Z_i$ for $i = 1, 2$.

We claim that $[G_1,F] \cap [G_2,F] = 1$.
By Corollary \ref{Fitting}, may assume that ${\bf F}(L/\Phi) \cong
E/F$ is a $\{ p, r \}$-subgroup.  Let $P/\Phi$ and $R/\Phi$ be
the Sylow $p$- and Sylow $r$-subgroups of ${\bf F}(L/\Phi)$,
respectively. Notice that $P \le L_1 \le G_1$ and $R \le L_2 \le
G_2$.
Again using the structure of $G_{1}$, we have that $1 \neq [F,P] =
[F,G_1]$ is an abelian $p'$-group and is irreducible under the
action of $P$ and so $[F,P,P] = [F,P]$.  Using Fitting's
lemma, ${\bf C}_{[F,P]} (P) = 1$.  This implies that $P$ does
not centralize any nonprincipal character in $\Irr
([F,P]\Phi/\Phi)$. In particular, $p$ divides $|L:I_L (\lambda)|$
for every nonprincipal $\lambda \in \Irr ([F,P]\Phi/\Phi)$.  By
Lemma \ref{stabs}, there does not exist a prime that
divides $|L:I_L (\lambda)|$ for every $\lambda \in \Irr
([F,E]\Phi/\Phi)$, so that $[F,P] = [F,G_1] < [F,E]$.  Similarly,
$1 < [F,R] = [F,G_2] < [F,E]$ and so $[F,P] \cap [F,R] <
[F,P]$.  Note that $P$ normalizes $[F,P] \cap [F,R]$.  The
irreducibility of the action of $P$ implies that $[F,P] \cap
[F,R] = 1$ and $[F,E] = [F,G_1] \times [F,G_2]$.

Note that $[F,G_1,G_2] \le [F,G_1] \cap [F,G_2] = 1$, so $G_2$
centralizes $[F,G_1]$ and $[F,G_1] \le Z_2$.  Hence, we have $Z_2 =
[F,G_1] \times (Z_1 \cap Z_2)$.  We conclude that $F = [F,G_1]
\times [F,G_2] \times Z$ where $Z = Z_1 \cap Z_2$ and it is not
difficult to see that $Z$ is the center of $G$.  We have $G = (L_1
L_2) F = (L_1 L_2) ([F,G_1] [F,G_2]) Z$.  Hence, $G = (L_1 [F,G_1]
Z)(L_2 [F,G_2] Z)$.  Now, we claim that $L_1 [F,G_1] Z$ and $L_2
[F,G_2] Z$ are normal subgroups of $G$.  To see this, observe that
both $L_2$ and $[F,G_2]$ normalize $L_1$, $[F,G_1]$, and $Z$.  Since
$L_1 [F,G_1] Z$ normalizes itself, $G$ normalizes it.
Similarly, $L_2 [F,G_2] Z$ is normal in $G$.

Let $N_1$ be a Hall $\{ p, q \}$-subgroup of $L_1 [F,G_1] Z$. Since
$Z$ is central in $G_1$, it follows that $H_ 1 = N_1 [F,G_1]$ is normal in
$L_1 [F,G_1] Z$.  Also, we know that $[F,G_1]$ is irreducible under
the action of $N_1$.  Since $N_1$ does not centralize $[F,G_1]$, we
deduce that $H_1 = {\bf O}^{\{p,q\}'}(L_1 [F,G_1] Z)$.  This implies $H_1$ is characteristic in $L_1 [F,G_1] Z$,
and so, $H_1$ is normal in $G$.  Similarly, if $N_2$ is a Hall $\{
r, s \}$-subgroup of $L_2 [F,G_2] Z$, then $H_2 = N_2 [F,G_2]$ is
normal in $G$.  Observe that $H_1 \cap Z = N_1 \cap Z$ is the Hall $\{ p ,q \}$-subgroup of $Z$ and $H_2 \cap Z = N_1 \cap Z$ is the Hall $\{ r, s \}$ subgroup of $Z$.  Taking $Y$ to be the Hall $\{ p, q, r, s \}$
complement of $Z$, we obtain $G = H_1 H_2 Z = H_1 H_2 Y$.  Notice that $H_1 \cap H_2 \le L_1 [F,G_1]Z \cap L_2 [F,G_2]Z = Z$, so $H_1 \cap H_2 = (H_1 \cap Z) \cap (H_2 \cap Z) = 1$.  Therefore, we conclude that $G = H_1
\times H_2 \times Y$.
\end{proof}

We come to the main theorem of this section.  Notice that this result combined with Theorem \ref{nonab p} proves the Main Theorem of the paper.  The main work left to prove this theorem is when $F$ is not abelian.

\begin{theorem}\label{Theorem 3.5}
Assume Hypothesis 3.  Then $G = M \times N$ where $\rho (M) = \{ p, q \}$ and $\rho (N) = \{ r, s \}$.
\end{theorem}

\begin{proof}
If $F$ is abelian, then the conclusion is Theorem \ref{Corollary 3.6}.  Thus, we may assume $F$ is not abelian.

Define $G_1$ and $G_2$ as in Theorem \ref{punchline}.  First, we note that $\rho(G_{i}) = \rho(F) \cup \pi(|G_{i}:F|)$ since $F(G_{i}) = F$.  It follows that $\{p, q \} \subseteq \rho(G_{1})$ and $\{r, s\} \subseteq \rho(G_{2})$, and so, both $\rho (G_{1})$ and $\rho (G_{2})$ contain at least two primes.  Recall that $G_{1}/F$ is the normal Hall $\{p, q\}$-subgroup of $G/F$, and $G_{2}/F$ is the normal Hall $\{r, s\}$-subgroup of $G/F$.

We claim that $\rho (G_{i}) \neq \rho (G)$, for $i = 1, 2$. To prove this, suppose $\rho (G_{1}) = \rho(G)$.  Since $G_1$ is normal in $G$, we see that $\Delta (G_1)$ is a subgraph of $\Delta (G)$, and so, $\Delta(G_{1}) = \Delta (G) = \Gamma$.  Since $\rho (G_1/\Phi (G)) = \{ p, q \}$, it must be that $G_1$ has normal, nonabelian Sylow $r$- and Sylow $s$-subgroups.  By Corollary \ref{two nonab} applied to $G_1$, we deduce that $h(G_{1}) = 2$, a contradiction since $G_1/\Phi (G)$ has Fitting height at least $3$.  Similarly, we also get $\rho (G_{2}) \neq \rho (G)$.

Suppose $G_{2}/\Phi(G)$ is disconnected of Type 2 or 3.  We see that
$\{ r, s \} = \{ 2, 3 \}$.  In this situation, we claim that both
$\rho (G_{1})$ and $\rho (G_{2})$ have two primes. We have seen that
$\rho (G_{i}) \neq \rho (G)$ for $i=1, 2$. Next, we show that $\rho
(G_{1})$ has three primes if and only if $\rho (G_{2})$ has three
primes.  Suppose first that $\rho (G_{2}) = \{2, 3, t \}$, where $t
\in \{p, q \}$.  Notice that this forces $F$ to have a nonabelian
Sylow $t$-subgroup.  Suppose $\rho(G_{1}) = \{p, q \}$, then $G_{1}$
is disconnected of Type 5, since ${\bf F}(G_{2}) = F = {\bf
F}(G_{1})$ is nonabelian.  This implies $2 \in \rho (G_{1})$,
contradicting $\{r,s\} \cap \{p,q\} = \emptyset$. Conversely,
suppose $\rho (G_1) = \{ p, q, e \}$ with $e \in \{ 2, 3 \}$, and
$\rho (G_2) = \{ 2, 3 \}$.  This forces $F$ to have a nonabelian
Sylow $e$-subgroup, but now $G_2$ is disconnected and has a
nonabelian Fitting subgroup and has disconnected group of Type 2 or
3 as a quotient, a contradiction.

Finally, suppose $\rho (G_{1}) = \{p, q, t \}$ and $\rho (G_{2}) =
\{2, 3, e \}$, where $t \in \{2, 3 \}$ and $e \in \{ p, q \}$. Now
consider $G_{1}/({\bf O}_{t}(G))^{'}$.  It is clear that
$\Delta(G_{1}/({\bf O}_{t}(G))^{'})$ has two connected components
$\{p \}$ and $\{q \}$.  And observe that the Fitting subgroup of
$G_{1}/{\bf O}_{t} (G)^{'}$ is nonabelian since ${\bf O}_{e}(G)$ is
nonabelian. This implies that $G_{1}/({\bf O}_{t}(G))^{'}$ is
disconnected of Type 5, again contradicting $\{r,s\} \cap \{p,q\} =
\emptyset$, and so the claim holds.

Now, we proved $\rho (G_{1}) = \{p, q \}$ and $\rho (G_{2}) = \{2, 3 \}$. Since $G_{2}/\Phi (G)$ is disconnected of Type 2 or 3, we conclude that $G_{2}$ is disconnected of Type 2 or 3.  But disconnected groups of Types 2 and 3 have abelian Fitting subgroups, and so, ${\bf F} (G_2) = F$ is abelian, a contradiction.

We now have $G_{1}/\Phi(G)$ and $G_{2}/\Phi(G)$ are disconnected of Type 4.  There are three possibilities that can
happen from the argument at beginning of the proof:
(a) $|\rho (G_{i})| = 3$ for both $i = 1, 2$, (b) $|\rho (G_{i})| = 2$ for both $i = 1, 2$, and (c) $|\rho (G_{1})| = 2$ and $|\rho (G_{2})| = 3$.

If case (b) happens, then ${\bf F}(G_{i}) = F$ is abelian, and this
is a contradiction.  (If $F$ has a nonabelian Sylow $t$-subgroup,
then $t$ will be in both $\rho(G_1)$ and $\rho (G_2)$.)  Thus, we
know that (b) does not occur.  Suppose that $\rho (G_{1}) = \{p, q,
t \}$ and $\rho (G_{2}) = \{r, s, e \}$, where $t \in \{r, s \}$ and
$e \in \{p, q \}$.  Now consider $G_{1}/({\bf O}_{t}(G))^{'}$ and
$G_{2}/({\bf O}_{e}(G))^{'}$.  Observe that $\Delta (G_{1}/({\bf
O}_{t} (G))^{'})$ has two connected components $\{p \}$ and $\{q \}$
and $\Delta (G_{2}/({\bf O}_{e} (G))^{'})$ has two connected
components $\{r \}$ and $\{s \}$.  Both $G_{1}/({\bf O}_{t}(G))^{'}$
and $G_{2}/({\bf O}_{e} (G))^{'}$ have nonabelian Fitting subgroups.
We conclude that both $G_{1}/({\bf O}_{t} (G))^{'}$ and $G_{2}/({\bf
O}_{e} (G))^{'}$ are disconnected of Type 5, a contradiction to
$\{p,\ q\} \cap \{r,\ s\}= \emptyset$. Thus, (a) does not occur.

Suppose case (c) happens.  Let $\rho (G_{1}) = \{ p, q\}$ and $\rho
(G_{2}) = \{r, s, e \}$, where $e \in \{p, q \}$.  We note that
$\Delta(G_{1})$ has two connected components $\{p \}$ and $\{q \}$.
Write $F = {\bf O}_{e} (G) \times Z$.  Since $F$ is nonabelian, it
is clear that $G_{1}$ is disconnected of Type 5, $Z\leq {\bf
Z}(G_{1})$, and $e = q = 2$.  Let $F_1$ and $E_1/F_1$ be the Fitting
subgroups of $G_1$ and $G_1/F_1$.  Since $G_1$ is of Type 5,
$|G_1:E_1| = 2$.

Let $G_3 = E_1 G_2$, and observe that $|G:G_3| = 2$.  Also, it is
not difficult to see that $\rho (G) = \rho (G_3)$, so $\Delta (G_3)
= \Delta (G)$.  Notice that $G_3$ has a nonabelian  normal Sylow
$2$-subgroup. Thus, we can appeal to Theorem \ref{nonab p} in $G_3$,
and we obtain $G_3 = K_1 \times K_2$ where $K_1$ and $K_2$ are characteristic subgroups of $G$ with $\rho
(K_1) = \{ p, 2 \}$ and $\rho (K_2) = \{ r, s \}$.  Thus, $K_1$ and
$K_2$ are normal in $G$. Notice that $G_2 = K_2 F$, and so $G_2/\Phi
(G) \cong K_2/(K_2 \cap \Phi (G))$, and so, $K_2$ is disconnected of
Type 4.

Let $\theta \in \Irr (K_1)$ have $p$ dividing $\theta (1)$.  Observe
that $\theta \times 1_{K_2}$ will be $G$-invariant, since $2p$
divides no degree in $\cd G$.  If $\gamma \in \Irr (K_2)$, then the
stabilizer of $\theta \times \gamma$ will be the stabilizer of
$\gamma$ in $G$.  Again, since $2p$ does not divide any degree in
$\cd G$, we conclude that $\gamma$ is $G$-invariant for all $\gamma
\in \Irr (K_2) = \Irr (G_3/K_1)$.  It follows that every character
in $\Irr (G_3/K_1)$ extends to $G$, and so, $\cd {G/K_1} = \cd
{G_3/K_1}$. Since $G_3/K_1 \cong K_2$ is disconnected of Type 4, we
see that $G/K_1$ is disconnected of Type 4.  It follows that $G/K_1
= M/K_1 \times G_3/K_1$ where $M/K_1$ is the Sylow $2$-subgroup of
the center of $G/K_1$.  We now have $G = M \times K_2$, and the
result follows by taking $N = K_2$.
\end{proof}

\vspace{0.3cm} \noindent{\normalsize Acknowledgements. } The
research of the second author was partially supported by China
Scholarship Council (CSC). The most work of this paper was done
while she was visiting Kent State University with the first author
and it will be one part of the second author's Ph.D thesis. She
thanks the Department of Mathematical Sciences of Kent State
University for its hospitality and appreciates Prof. Lewis' guidance.



\end{document}